\documentclass[]{interact}

\usepackage{epstopdf}
\usepackage[caption=false]{subfig}

\usepackage[numbers,sort&compress]{natbib}
\bibpunct[, ]{[}{]}{,}{n}{,}{,}
\renewcommand\bibfont{\fontsize{10}{12}\selectfont}

\usepackage{lineno,hyperref}
\usepackage{amsmath,amssymb} 
\usepackage{amsthm,enumerate,verbatim}
\usepackage{amsfonts}
\usepackage{mathabx}
\usepackage{graphicx}
\usepackage{url}
\usepackage{subfig}
\usepackage{mathtools}
\usepackage[toc,page]{appendix}
\usepackage{multirow}
\usepackage{xcolor}
\usepackage{hyperref}
\usepackage{algorithm}
\usepackage{algorithmic}

\theoremstyle{plain}
\newtheorem{theorem}{Theorem}[section]

\newtheorem{corollary}[theorem]{Corollary}

\theoremstyle{definition}

\theoremstyle{remark}
\newtheorem{remark}{Remark}

\DeclareMathOperator{\rank}{rank} 

\definecolor{brightpink}{rgb}{1.0, 0.0, 0.5}

\newcommand{\ngi}[1]{{{\color{black} #1}}}

\newcommand{\vlepi}[1]{{{\color{black} #1}}}

\newcommand{\revise}[1]{{{\color{black} #1}}}

\begin{document}


\title{\ngi{Conic-Optimization Based Algorithms for \\ 
 Nonnegative Matrix Factorization}}

\author{
\name{
Valentin Leplat\textsuperscript{a}\thanks{VL acknowledges the support by the European Research Council (ERC Advanced Grant no 788368) and the support by Ministry of Science and Higher Education grant No. 075-10-2021-068. Email: V.Leplat@skoltech.ru}, 
Yurii Nesterov\textsuperscript{b}, 
Nicolas Gillis\textsuperscript{c}\thanks{NG acknowledges the support by the Fonds de la Recherche Scientifique - FNRS and the Fonds Wetenschappelijk Onderzoek - Vlanderen (FWO) under EOS Project no O005318F-RG47, by the European Research Council (ERC Starting Grant no 679515), and by the Francqui Foundation. Email: nicolas.gillis@umons.ac.be},  
François Glineur\textsuperscript{b}\thanks{YN and FG acknowledge support from the European Research Council (ERC) under the European Union’s Horizon 2020 research and innovation program (grant agreement No. 788368). FG also acknowledges support from EOS project no O005318F-RG47. Emails: \{yurii.nesterov,francois.glineur\}@uclouvain.be.}
}
\affil{\textsuperscript{a}Center for Artificial Intelligence Technology, Skoltech, Bolshoy Boulevard 30, bld. 1, Moscow, Russia 121205; \textsuperscript{b}CORE and ICTEAM/Mathematical Engineering (INMA), UCLouvain, Avenue Georges Lemaître 4, B-1348 Louvain-la-Neuve, Belgium ; \textsuperscript{c}Department of Mathematics and Operational Research,
Facult\'e Polytechnique, Universit\'e de Mons,
Rue de Houdain 9, 7000 Mons, Belgium}
}

\maketitle

\begin{abstract}
\revise{Nonnegative matrix factorization} is the following problem: given a nonnegative input  matrix $V$  and a factorization rank $K$,  compute two nonnegative matrices, $W$ with $K$ columns and $H$ with $K$ rows, such that $WH$ approximates $V$ as well as possible. In this paper, we propose two new approaches for computing high-quality NMF solutions  using conic optimization. These approaches rely on the same two steps. First, we reformulate NMF as minimizing a concave function over a product of convex cones--one approach is based on the exponential cone, and the other on the second-order cone. Then, we solve these reformulations iteratively: at each step, we minimize exactly, over the feasible set, a majorization of the objective functions obtained via linearization at the current iterate. Hence these subproblems are convex conic programs and can be solved efficiently using dedicated algorithms. We prove that our approaches reach a stationary point with an accuracy decreasing as  $\mathcal{O}(\frac{1}{i})$, where $i$ denotes the iteration number. To the best of our knowledge, our analysis is the first to provide a convergence rate to stationary points for NMF. Furthermore, in the particular cases of rank-one factorizations (that is, $K=1$), we show that one of our formulations can be expressed as a convex optimization problem implying that optimal rank-one approximations can be computed efficiently. Finally, we show on several numerical examples that our approaches are able to frequently compute exact NMFs (that is, with $V = WH$), and compete favorably with the state of the art. 
\end{abstract}


\section{Introduction}

Nonnegative matrix factorization (NMF) 
is the problem of approximating a given nonnegative matrix, $V \in \mathbb{R}_+^{F \times N}$, as the product of two smaller nonnegative matrices, $W \in \mathbb{R}_+^{F \times K}$ and $H \in \mathbb{R}_+^{K \times N}$, where $K$ is a given parameter known as the factorization rank. One aims at finding the best approximation, that is, the one that minimizes the discrepancy between $V$ and the product $WH$, often measured by the Frobenius norm of their difference, $\left\| V-WH \right\|_F$.  
Despite the fact that NMF is NP-hard in general~\cite{vavasis2009complexity} (see also~\cite{doi:10.1137/16M1080999}), it has been used successfully in many domains such as probability, geoscience, medical imaging, computational geometry, combinatorial optimization, analytical chemistry, and machine learning; see~\cite{fu2019nonnegative, gillis2020bk} and the references therein. 
Many local optimization schemes have been developed to compute NMFs. They aim to identify local minima or stationary points of optimization problems that minimize the discrepancy between $V$ and the approximation $WH$. Most of these iterative algorithms rely on a two-block coordinate descent scheme that consists in (approximatively) optimizing alternatively over $W$ with $H$ fixed, and vice-versa; 
see~\cite{cichocki2009nonnegative, gillis2020bk} and the references therein.  
In this paper, we are interested in computing high-quality local minima for the NMF optimization problems without relying on the block coordinate descent (BCD) framework. We will perform the optimization over $W$ and $H$ jointly. 
Moreover, our focus is on finding exact NMFs, that is, computing nonnegative factors $W$
and $H$ such that $V=WH$, although our approaches can be used to find approximate NMFs as well.  

The minimum
factorization rank $K$ for which an exact NMF exists is called the nonnegative rank of $V$ and is
denoted $\rank_+(V)$, we have 
\begin{equation*}
\text{rank}_+(V)=\text{min}\Big\{  K \in \mathbb{N} \, \Big| \, 
		\text{there exist } W \in \mathbb{R}_+^{F \times K} \text{ and } H \in \mathbb{R}_+^{K \times N} \text{ such that } V = WH \Big\}.  
\end{equation*} 
The computation of the nonnegative rank is NP-hard~\cite{vavasis2009complexity},  and is a research topic on its own;  
see~\cite[Chapter 3]{gillis2020bk}\cite{Dewez2020, dewez2022computational} 
and the references therein for recent progress on this question.

\subsection{Computational complexity}

Solving exact NMF can be used to compute the nonnegative rank, by finding the smallest $K$ such that an exact NMF exists. 
Cohen and Rothblum~\cite{cohen1993nonnegative} give a super-exponential time algorithm for this problem. Vavasis~\cite{vavasis2009complexity} proved that checking whether $\rank(V)=\rank_+(V)$, where $\rank(V)=K$ is part of the input, is NP-hard. 
Since determining the nonnnegative rank is a generalization of exact NMF, the results in~\cite{vavasis2009complexity} imply that computing an exact NMF is also NP-hard. 
Similarly, the standard NMF problem using any norm is a generalization of exact NMF, and therefore any hardness result that applies to exact NMF also applies to most approximated NMF models~\cite{vavasis2009complexity}. 
Hence, 
unless \textbf{P}=\textbf{NP}, no algorithm can solve exact NMF using a number of arithmetic operations bounded by a polynomial in $K$ and in the size of $V$; see also~\cite{doi:10.1137/16M1080999} that gives a different proof using algebraic arguments. 

More recently, Arora et al.~\cite{doi:10.1137/130913869} showed that no algorithm to solve this problem can run in time $(FN)^{o(K)}$ unless\footnote{3-SAT, or 3-satisfiability, is an instrumental problem in computational complexity to prove NP-completeness results. 3-SAT is the problem of deciding whether a set of disjunctions containing 3 Boolean variables or their negation can be satisfied.} 3-SAT can be solved in time $2^{o(n)}$ on instances with $n$ variables. However, in practice, $K$ is small and it makes senses to wonder what is the complexity if $K$ is assumed to be a  fixed constant. In that case, they showed that exact NMF can be solved in polynomial time in $F$ and $N$, namely in time $\mathcal{O}((FN)^{c 2^K K^2})$ for some constant $c$, which Moitra~\cite{Moit13} later improved to  $\mathcal{O}((FN)^{cK^2})$.  

The argument is based on quantifier elimination theory, using the seminal result by Basu, Pollack and Roy~\cite{basu1996}. 
Unfortunately, this approach cannot be used in practice, even for small size matrices, because of its high computational cost: although the term $\mathcal{O}((FN)^{cK^2})$ is a polynomial in $F$ and $N$ for $K$ fixed, it grows extremely fast (and the hidden constants are usually very large). Let us illustrate with a 4-by-4 matrix with $K=3$,  we have a complexity of order $16^9 \approx 7 \ 10^{10}$ and for a 5-by-5 matrix with $K=4$, the complexity raises up to $25^{16} \approx 2 \ 10^{22}$. 
Therefore developing an effective computational technique for exact NMF of small matrices is an important direction of research. Some heuristics have been recently developed that allow solving exact NMF for matrices up to a few dozen rows and columns~\cite{vandaele2016}.

\subsection{Contribution and outline of the paper}

In this paper, we introduce two \revise{formulations} for computing an NMF using conic optimization. They rely on the same two steps. 
First, in Section~\ref{exactNMFcriteria}, we reformulate NMF as minimizing a concave function over a product of convex cones; one approach is based on the exponential cone \revise{and leads to under-approximations}, and the other on the second-order cone \revise{and leads to over-approximations}. For the latter formulation, in the  case of a rank-one factorization, we show that it can be cast as a convex optimization problem, leading to an efficient computation of the \revise{optimal rank-one over-approximation}. Then, in Section~\ref{sec6_algo}, we solve these reformulations iteratively: at each step, we minimize exactly over the feasible set a majorization of the objective functions  obtained via linearization at the current iterate. Hence these subproblems are convex conic programs and can be solved efficiently using dedicated algorithms. In Section~\ref{sec_convRes}, we show that our optimization scheme relying on successive linearizations is a special case of the Frank-Wolfe (FW) algorithm. By using an appropriate measure of stationarity, namely the FW gap, 
we show in Theorem~\ref{Theo1} 
that the minimal FW gap generated by our algorithm  converges as  $\mathcal{O}(\frac{1}{i})$, where $i$ is the iteration index.   
Finally, in Section~\ref{num_exp}, \revise{we use our approaches to  compute
exact NMFs, and show that they compete} favorably with the state of the art when applied to several classes of nonnegative matrices; namely, randomly generated, infinitesimally rigid and slack matrices.

\revise{
\begin{remark}[Focus on exact NMF]
Our two NMF formulations can be used to compute approximate factorizations (namely, under- and over-approximations). However, in this paper, we focus on exact NMF for the numerical experiments. 
The reason is twofold:  
\begin{enumerate}
    \item Exact NMF problems allow us to guarantee whether a globally optimal solution is reached, and hence compare algorithms in terms of global optimality. 
    
    \item Our current algorithms rely on interior-point methods that do not scale well. 
Therefore, they are significantly slower, at this point, than  state-of-the-art algorithms to compute approximate NMFs on large data sets (such as images or documents). Making our approach scalable is a topic of further research. 
Moreover, because of the under/over-approximations, the error obtained with the proposed algorithms would be larger, and the comparison would not be fair. Here is a simple example, 
let 
\[
V = \left( 
\begin{array}{cc}
    0 & 1 \\
    1 & 1
\end{array} \right), 
V_{\text{NMF}} = 
\left( 
\begin{array}{cc}
    0.45  &  0.72 \\ 
    0.72  &  1.17
\end{array} \right), 
V_{\text{U}} = 
\left( 
\begin{array}{cc}
    0  & 1 \\ 
    0  &  1
\end{array} \right), 
V_{\text{O}} = 
\left( 
\begin{array}{cc}
    1  & 1 \\ 
    1  &  1
\end{array} \right). 
\]
The best rank-one NMF of $V$ is $V_{\text{NMF}}$ (with two digits of accuracy) with Frobenius error 
$\| V -  V_{\text{NMF}} \|_F = 0.62$, while a best rank-one under-approximation (resp.\ over-approximation) of $V$ is $V_{\text{U}}$ (resp.\ $V_{\text{O}}$) with Frobenius error 1. (Note however that under-/over-approximations can have some useful properties in practice~\cite{gillis2010using, tepper2017nonnegative}.) 
\end{enumerate} 
\end{remark}
}


\section{\revise{NMF formulations} based on conic optimization} \label{exactNMFcriteria}

In this section we propose two new formulations for NMF, where the feasible set is represented using the exponential cone  (Section~\ref{sec:expon}) and the second-order cone (Section~\ref{sec:escondorder}).

\subsection{NMF formulation via exponential cones} \label{sec:expon}

Given a non-negative matrix $V \in \mathbb{R}_{+}^{F \times N}$ and a positive integer $K \ll \min(F,N)$, we want to compute an  NMF. Our first proposed formulation is the following: 
\begin{equation}\label{basicNMFoptiprob}
		\begin{aligned}
		& \underset{W\in \mathbb{R}^{F \times K}, 
		H\in \mathbb{R}^{K \times N}}{\text{max}}
		& & \sum_{f=1}^F \sum_{n=1}^N \left(\sum_{k=1}^K W_{fk}H_{kn} \right) \\
		& \text{subject to}
		& & \sum_{k=1}^K W_{fk}H_{kn}   \leq  V_{fn} \text{ for } f \in \mathcal{F}, n \in \mathcal{N}, \\
		& \text{}
		& & W_{fk} \geq 0, H_{kn} \geq 0 \text{ for } f \in \mathcal{F}, k \in \mathcal{K}, n \in \mathcal{N}.
		\end{aligned}
	\end{equation}
where $\mathcal{F}=\{1,...,F \}$, $\mathcal{N}=\{1,...,N \}$ and $\mathcal{K}=\{1,...,K \}$. 
Any feasible solution $(W,H)$ of~\eqref{basicNMFoptiprob}  provides an under-approximation of $V$, because of the elementwise constraint $WH \leq V$. 
The objective function of~\eqref{basicNMFoptiprob} maximizes the sum of the entries of $WH$. Therefore, if $V$ admits an exact NMF of size $K$, that is, $\rank_+(V) \leq K$, any optimal solution $(W^*,H^*)$ of~\eqref{basicNMFoptiprob} must satisfy $W^*H^*=V$, and hence will provide an exact NMF of $V$. 
Note that this problem is nonconvex because of the bilinear terms appearing in the objective and the constraint $WH \leq V$.

Let us now reformulate~\eqref{basicNMFoptiprob} using exponential cones. 
In order to deal with nonnegativity constraints on the entries of $W$ and $H$, we use the following change of variables: $W_{fk}=G(U_{fk})=e^{U_{fk}}$ and $H_{kn}=G(T_{kn})=e^{T_{kn}}$, where $U \in \mathbb{R}^{F \times K}$ and $T \in \mathbb{R}^{K \times N}$, with $f=1,\dots,F$, $n=1,\dots,N$ and $k=1,\dots,K$ and $G(t)=e^{t}$. 
By applying a logarithm on top of this change of variables to the objective function, and on both sides of the inequality constraints $WH \leq V$,   \eqref{basicNMFoptiprob} can be nearly equivalently rewritten as follows,  the difference being that zero elements in $W$ and $H$ are now excluded:  
\begin{equation}\label{optiprob1}
		\begin{aligned}
		& \underset{U\in \mathbb{R}^{F \times K},T\in \mathbb{R}^{K \times N}}{\text{max}}
		& & \text{log}\left(\sum_{f,n,k}  e^{U_{fk}+T_{kn}} \right) \\
		& \text{subject to}
		& & \text{log}\left( \sum_{k=1}^K e^{U_{fk}+T_{kn}} \right) \leq \text{log}\left( V_{fn} \right) \text{ for } f \in \mathcal{F}, n \in \mathcal{N},
		\end{aligned}
	\end{equation} 
which corresponds to the maximization of a convex function (logarithm of the sums of exponentials) over a convex set, each constraint being convex for the same reason. 

We rewrite the convex feasible set of \eqref{optiprob1}  
with explicit conic constraints as follows:
	\begin{equation}\label{logsummodel_setQ}
        \begin{aligned}
    		& \sum_{k=1}^{K} t_{fkn} \leq V_{fn} \text{ for } f \in \mathcal{F}, n \in \mathcal{N},\\
    		& \left( t_{fkn},1,U_{fk}+T_{kn}\right) \in K_{exp} \text{ for } f \in \mathcal{F}, k \in \mathcal{K}, n \in \mathcal{N},
        \end{aligned}
    \end{equation}
where $K_{exp} \subset  \mathbb{R}^3$ denotes the (primal) exponential cone defined as:
\begin{equation}\label{def:expoCo}
    \begin{aligned}
		& K_{exp} = \left\{ (x_1,x_2,x_3) \in \mathbb{R}^3 | x_1 \geq x_2 e^{\frac{x_3}{x_2}}, x_2 > 0   \right\} \cup 
		 \left\{ \left( x_1,0,x_3 \right) | x_1 \geq 0, x_3 \leq 0  \right\}.
    \end{aligned}
\end{equation}
Note that the exponential cone is closed and includes the subset $\left\{ \left( x_1,0,x_3 \right) | x_1 \geq 0, x_3 \leq 0  \right\}$, therefore the scenarios for which the entries $V_{fn}$ are equal to zero can be handled by exponential conic constraints, which was not possible with formulation \eqref{optiprob1} since the $\log$ function is not defined at zero. 
Hence the optimization problem \eqref{basicNMFoptiprob} can be written completely  equivalently as  
\begin{equation}\label{optiprob2}
		\begin{aligned}
		& \underset{U\in \mathbb{R}^{F \times K}, 
		T\in \mathbb{R}^{K \times N}, 
		t \in \mathbb{R}^{F \times K \times N}}{\text{max}}
		& & \text{log}\left(\sum_{f,n,k} e^{U_{fk}+T_{kn}} \right) \\
		& \text{subject to}
		& & \sum_{k=1}^{K} t_{fkn} \leq V_{fn} \text{ for } f \in \mathcal{F}, n \in \mathcal{N} ,\\
		& \text{}
		& & \left( t_{fkn},1,U_{fk}+T_{kn}\right) \in K_{exp} \text{ for } f \in \mathcal{F}, k \in \mathcal{K}, n \in \mathcal{N}.
		\end{aligned}
	\end{equation}
This leads to $F \times N$  inequality constraints and the introduction of $F \times K \times N$ exponential cones. 
In Section~\ref{sec6_algo}, we propose an algorithm to tackle~\eqref{optiprob2} using successive linearizations of the objective function.

\subsection{NMF formulation via rotated second-order cones} \label{sec:escondorder} 

Our second proposed \revise{NMF} formulation is the following: 
\begin{equation}\label{basicNMFoptiprob2}
		\begin{aligned}
		& \underset{W\in \mathbb{R}^{F \times K},H\in \mathbb{R}^{K \times N}}{\text{min}}
		& & \sum_{f=1}^F \sum_{n=1}^N \left(\sum_{k=1}^K W_{fk}H_{kn} \right) \\
		& \text{subject to}
		& & \sum_{k=1}^K W_{fk}H_{kn}   \geq  V_{fn} \text{ for } f \in \mathcal{F}, n \in \mathcal{N}, \\
		& \text{}
		& & W_{fk}, H_{kn} \geq 0 \text{ for } f \in \mathcal{F}, k \in \mathcal{K}, n \in \mathcal{N}.
		\end{aligned}
	\end{equation}
Any feasible solution $(W,H)$ of~\eqref{basicNMFoptiprob2}  provides an over-approximation of $V$, because of the constraint $WH \geq V$. 
The objective function of~\eqref{basicNMFoptiprob} minimizes the sum of the entries of $WH$. Therefore, if $\rank_+(V) \leq K$, any optimal solution $(W^*,H^*)$ of~\eqref{basicNMFoptiprob} must satisfy $W^*H^*=V$, and hence will provide an exact NMF of $V$. Again the problem is nonconvex due to the bilinear terms.

Let us use the following change of variables: we let  $W_{fk}=G(U_{fk})=\sqrt{U_{fk}}$ and $H_{kn}=G(T_{kn})=\sqrt{T_{kn}}$ 
where $U \in \mathbb{R}_+^{F \times K}$ and $T \in \mathbb{R}_+^{K \times N}$, with $f=1,\dots,F$, $n=1,\dots,N$ and $k=1,\dots,K$, this time with $G(t)=\sqrt{t}$. 
Thus the optimization problem \eqref{basicNMFoptiprob2} can be equivalently rewritten as:  
\begin{equation}\label{optiprob3}
		\begin{aligned}
		& \underset{U\in \mathbb{R}_+^{F \times K},T\in \mathbb{R}_+^{K \times N}}{\text{min}}
		& & \sum_{f=1}^F \sum_{n=1}^N \left( \sum_{k=1}^K \sqrt{U_{fk}}\sqrt{T_{kn}}\right) \\
		& \text{subject to}
		& & \sum_{k=1}^K \sqrt{U_{fk}}\sqrt{T_{kn}}  \geq  V_{fn}  \text{ for } f \in \mathcal{F}, n \in \mathcal{N},
		\end{aligned}
	\end{equation}
which minimizes a concave function over a convex set. Indeed, the function $\sqrt{xy}$ is concave. 

This set can be written with conic constraints as follows: 
	\begin{equation}\label{quadmodel_setQ}
        \begin{aligned}
    		& \sum_{k=1}^{K} t_{fkn} \geq V_{fn}, \text{ for } f \in \mathcal{F}, n \in \mathcal{N},\\
    		& \left( U_{fk},\frac{1}{2} T_{kn},t_{fkn}\right) \in \mathcal{Q}_r^{3}\text{ for } f \in \mathcal{F}, k \in \mathcal{K}, n \in \mathcal{N},
        \end{aligned}
    \end{equation}
where $\mathcal{Q}_r^{3}$ denotes the $3$-dimensional rotated second-order cone defined as: 
\begin{equation*}
    \begin{aligned}
		& \mathcal{Q}_{r}^{3} 
		= 
		\left\{ (x_1,x_2,x_3) \in \mathbb{R}^{3}\text{ }|\text{ }2 x_{1} x_{2} \geq x_{3}^{2}, x_1 \geq 0, x_2 \geq 0  \right\}.
    \end{aligned}
\end{equation*}

Thus, the optimization problem \eqref{optiprob3} becomes
\begin{equation}\label{optiprob4}
		\begin{aligned}
		& \underset{U\in \mathbb{R}_+^{F \times K},T\in \mathbb{R}_+^{K \times N},t \in \mathbb{R}^{F \times K \times N}}{\text{min}}
		& & \sum_{f=1}^F \sum_{n=1}^N \left( \sum_{k=1}^K \sqrt{U_{fk}}\sqrt{T_{kn}}\right) \\
		& \text{subject to}
		& & \sum_{k=1}^{K} t_{fkn} \geq V_{fn}  \text{ for } f \in \mathcal{F}, n \in \mathcal{N}, \\
		& \text{}
		& & \left( U_{fk},\frac{1}{2} T_{kn},t_{fkn}\right) \in \mathcal{Q}_r^{3}\text{ for } f \in \mathcal{F}, k \in \mathcal{K}, n \in \mathcal{N}, 
		\end{aligned}
	\end{equation}
which leads to $F \times N$  inequality constraints and the introduction of of $F \times K \times N$ rotated quadratic cones. 
In section~\ref{sec6_algo}, we present an algorithm to tackle \eqref{optiprob2} and \eqref{optiprob4}.

\subsubsection{Rank-one Nonnegative Matrix Over-approximation}\label{subsec_rank1NMO}

\vlepi{In this section, we show that our over-approximation formulation \eqref{basicNMFoptiprob2} can be expressed as a convex optimization problem in the case of a rank-one factorization (that is, $K=1$). 
Hence we will be able to compute an optimal rank-one nonnegative matrix over-approximation (NMO). 
For $K=1$, \eqref{basicNMFoptiprob2} becomes: 
\begin{equation}\label{rank1NMO1}
		\begin{aligned}
		& \underset{w\in \mathbb{R}^{F},h\in \mathbb{R}^{N}}{\text{min}}
		& & \sum_{f=1}^F \sum_{n=1}^N  w_{f}h_{n}  \\
		& \text{subject to}
		& & w_{f}h_{n}  \geq  V_{fn} \text{ for } f \in \mathcal{F}, n \in \mathcal{N}, \\
		& \text{}
		& & w_{f}, h_{n} \geq 0 \text{ for } f \in \mathcal{F}, n \in \mathcal{N}.
		\end{aligned}
	\end{equation}
Any feasible solution $(w,h)$ of~\eqref{rank1NMO1}  provides a rank-one NMO of $V$, because of the constraints. 
The objective function of~\eqref{rank1NMO1} minimizes the sum of the entries of $wh^\top$, which is equal to 
$\langle wh^\top,e e^\top \rangle  
= \langle w,e \rangle \ 
\langle h,e \rangle$, 
where $e$ denotes the all-one factor of appropriate dimension. 
Since any solution $wh^\top$ can be rescaled as $(\lambda w)(h^\top /\lambda)$ for any $\lambda > 0$, we can assume without loss of generality (w.l.o.g.) that $\langle w,e \rangle = 1$, and hence 
\eqref{rank1NMO1} can be equivalently written as follows: 
\begin{equation}\label{rank1NMO2}
		\begin{aligned}
		& \underset{w\in \mathbb{R}^{F},h\in \mathbb{R}^{N}}{\text{min}}
		& & \langle h,e \rangle  \\
		& \text{subject to}
		& & \langle w,e \rangle = 1, \\
		& \text{}
		& & w_{f}h_{n}  \geq  V_{fn} \text{ for } f \in \mathcal{F}, n \in \mathcal{N}, \\
		& \text{}
		& & w_{f}, h_{n} \geq 0 \text{ for } 
		f \in \mathcal{F}, n \in \mathcal{N}. 
		\end{aligned}
	\end{equation}
 Then, by letting each $h_n$ take the minimal value allowed by the constraints, that is, $h_n = \text{max}_{f \in \mathcal{F}} V_{fn}/w_f$ for each $n \in \mathcal{N}$, and replacing $w_f$ by its inverse, $u_f=1/w_f$ for each $f \in \mathcal{F}$,
 \eqref{rank1NMO2} becomes: 
\begin{equation}\label{rank1NMO3}
		\begin{aligned}
		& \underset{u\in \mathbb{R}^{F}}{\text{min}}
		& & \sum_n \text{max}_f \left(u_f V_{fn}  \right)  \\
		& \text{subject to}
		& & \sum_f 1/u_f \leq 1, \; u_{f} \geq 0 \text{ for } f \in \mathcal{F}. 
		\end{aligned}
	\end{equation} 
The feasible set of \eqref{rank1NMO3} can be formulated by using various conic constraints:
\begin{itemize}
    \item a semi-definite programming formulation: introduce variables $y_f$ such that $u_f y_f \geq 1, \sum_f y_f \leq 1$ to obtain 
    \begin{equation*}
    \begin{aligned}
                & \left(
                    \begin{array}{cc}
                        u_f &    1 \\
                        1 &    y_f  
                    \end{array} 
                 \right) \in \mathbb{S}^{2}_+ \text{ for } f \in \mathcal{F}, \; \sum_f y_f \leq 1, 
    \end{aligned}
    \end{equation*}
    where $\mathbb{S}^{2}_+$ denotes the set of positive semi-definite matrices of dimension 2.
    \item a power-cone formulation:  for $p<0$ the function $g(x)=x^{p}$ is convex for $x >0$ and the inequality $z \geq x^{p}$ is equivalent to $z^{1/(1-p)}x^{-p/(1-p)} \geq 1 \iff (z,x,1) \in P^{1/(1-p)}$ where $P^\alpha = \{ (x,z,a) \ | \ z^\alpha x^{1-\alpha} \geq a\}$ is a power cone. In our case, by introducing $y_f \geq u_f^{-1}$, we obtain  
    \begin{equation*}
        \begin{aligned}
      &(y_f,u_f,1) \in P^{1/2} \text{ for } f \in \mathcal{F}, \; \sum_f y_f \leq 1. 
        \end{aligned}
    \end{equation*}
    \item a (rotated) quadratic formulation:  introducing variables $y_f$ such that $y_f \geq 1/u_f$ for $u_f \geq 0$, this can be formulated  as follows: $(u_f,y_f,\sqrt{2}) \in Q_r^{3}$ where $Q_r^{3}$ denotes the set of rotated quadratic cones of dimension 3. We then have :  
    \begin{equation*}
        \begin{aligned}
                     &(u_f,y_f,\sqrt{2}) \in Q_r^{3} \text{ for } f \in \mathcal{F}, \ \; \sum_f y_f \leq 1 . 
        \end{aligned}
    \end{equation*}
\end{itemize}
In this paper, we consider the (rotated) quadratic formulation which is the easiest to implement in the MOSEK software \cite{mosek}. \\

Further, the objective function of~\eqref{rank1NMO3} is a sum of convex piece-wise linear functions. Hence by posing $t_n \geq \text{max}_f \left(u_f V_{fn}  \right)$, \eqref{rank1NMO3} can equivalently be formulated as follows:
 \begin{equation}\label{rank1NMO4}
		\begin{aligned}
		& \underset{t \in \mathbb{R}^{N}, u,y\in \mathbb{R}^{F}}{\text{min}}
		& & \sum_n t_n  \\
		& \text{subject to}
		& & \sum_f y_f \leq 1, \\
		& \text{}
		& & (u_f,y_f,\sqrt{2}) \in Q_r^{3} \text{ for } f \in \mathcal{F}, \\
		& \text{}
		& & t_n \geq u_f V_{fn} \text{ for } f \in \mathcal{F}, 
		n \in \mathcal{N}, 
		\end{aligned}
	\end{equation} 
which involves $2F+N$ variables and $F(1+N)+1$ constraints. 
This problem can be solved to optimality and efficiently with an interior-point method (IPM), as available for example in MOSEK~\cite{mosek}. \\ 
}

\ngi{
For the formulation based on exponential cones, \cite{NGillisMscThesis} showed that the rank-one underapproximation for positive input matrices can be expressed as the dual of an optimal transportation problem, and hence can also be solved optimally and efficiently with polynomial-time methods~\cite{SHARMA2000611}.
} 

\section{A successive linearization algorithm}\label{sec6_algo}

In this section, we present an iterative algorithm to tackle problems \eqref{optiprob2} and \eqref{optiprob4}. Both problems can be written as the minimization of a concave function $\Phi$ over a convex set denoted by $Q$. Note that $Q$ designates either the feasible set of \eqref{optiprob2} or the feasible set of \eqref{optiprob4}. 
 We perform this minimization by solving  a sequence of simpler problems in which the objective function is replaced by its linearization constructed at the current solution $(U,T)$. Let us denote $Z^{(i)}=(U^{(i)},T^{(i)})$ the $i$th iterate of our algorithm. 
 At each iteration $i$, we update $Z$ as follows: 
\begin{equation}\label{subproblm_exacNMF}
\begin{aligned}
    Z^{(i)} & \in \underset{Z \in Q}{\text{argmin }} \Phi(Z^{(i-1)}) + \langle \nabla \Phi(Z^{(i-1)}), Z-Z^{(i-1)}\rangle \\
          & \in \underset{Z \in Q}{\text{argmin }} \langle \nabla \Phi(Z^{(i-1)}), Z\rangle, 
\end{aligned}
\end{equation}  
where $\Phi$ is the objective function of \eqref{optiprob2} or  \eqref{optiprob4}. 
Since the objective of \eqref{subproblm_exacNMF} is linear in $Z$,  the subproblems become convex. Moreover they are particular structured conic optimization problems. 
In this paper, we use the MOSEK software \cite{mosek} to solve each successive problem \eqref{subproblm_exacNMF} with an IPM.  
Algorithm~\ref{SCCAE-NMF} summarizes our proposed method to tackle   \eqref{optiprob2} and \eqref{optiprob4}.   

To initialize $U$ and $T$, we chose to randomly initialize $W$ and $H$ (using the uniform distribution in the interval $[0,1]$ for each entry of $W$ and $H$) and apply the two changes of variables, $G(.)$, to compute the initializations for $U$ and $T$.
    
In this paper, we use a tolerance for the relative error equal to $10^{-6}$, that is, 
we assume that an exact NMF $(W,H)$ is found for an input matrix $V$ as soon as $\frac{\left\|V-WH \right\|_F}{\left\| V \right\|_F} \leq 10^{-6}$, as done in~\cite{vandaele2016}.  

\vlepi{
The main algorithm integrates a procedure that automatically updates the optimization problems in the case subsets of entries of the solution tend to zero. Indeed, due to numerical limitations of the solver, the required level of accuracy cannot be reached in some numerical tests even if the solution is close to convergence. This procedure is referred to as Sparsity Patterns Integration (SPI) and is detailed in Appendix~\ref{appB}. Note that the update of the optimization problem is computationally costly, in particular the update of the matrix of coefficients defining the constraints. Hence, SPI is triggered twice; at 80\% and 95\% of the maximum number of iterations. This practical choice has been motivated by numerical experiments that showed that two activations in the final iterations are sufficient to reach the tolerance error when the current solution is close enough to a high-accuracy local optimum. 
However, for exponential cones, in some numerical tests, the relative error can be stuck in the interval $[10^{-4},10^{-5}]$. In this context only, a final refinement step further improves the output of the main algorithm using the state-of-the-art accelerated HALS algorithm, an exact BCD method for NMF, from~\cite{6797154} to go below the $10^{-6}$ tolerance.
}

\algsetup{indent=2em}
\begin{algorithm}[ht!]
\caption{Successive  Conic Convex Approximation for Exact NMF  \label{SCCAE-NMF}}
\begin{algorithmic}[1] 
\REQUIRE Input matrix $V \in \mathbb{R}^{F \times N}_+$, the factorization rank $K$, 
 number of iterations \textit{maxiter}, \revise{choose the formulation~\eqref{optiprob2} (exponential cones) or~\eqref{optiprob4} (second-order cones).}  
\ENSURE $(W,H) \geq 0$ \revise{such that 
$V \approx WH$,
and 
$V \leq WH$ for~\eqref{optiprob2} (under-approximation) 
or  
$V \geq WH$ for~\eqref{optiprob4}  (over-approximation).} 

    \medskip 
    
\STATE \emph{\% Block 1: Initialization}
\STATE $(W^{(0)},H^{(0)})\longleftarrow$ positive random initialization$(F,K,N)$. 
\STATE $(U^{(0)},T^{(0)}) \longleftarrow G^{-1}(W^{(0)},H^{(0)})$  \text{ where } $G$ is the change of variables 
\STATE $Z^{(0)} \longleftarrow (U^{(0)},T^{(0)})$
\STATE \emph{\% Block 2: iterative update of $Z$}
\FOR {$i=1,2,\dots,$ maxiter}
    \STATE $Z^{(i)} \longleftarrow \underset{Z \in Q}{\text{argmin }} \langle \nabla \Phi(Z^{(i-1)}), Z\rangle$ with IPMs available in MOSEK \cite{mosek} 
\ENDFOR
\STATE $(W,H) \longleftarrow G(Z^{(i)})$

\end{algorithmic}
\end{algorithm}   
\vlepi{In Section~\ref{sec_convRes} 
, we discuss the convergence guarantees for Algorithm~\ref{SCCAE-NMF}.

\section{Convergence results}\label{sec_convRes}

Let us focus on the following optimization problem 
\begin{equation}\label{opt_prob}
		\begin{aligned}
		& \underset{Z \in Q}{\text{min}}
		& & \Phi\left(Z\right), 
		\end{aligned}
	\end{equation}
where $\Phi$ is a concave continuously differentiable function over the domain $Q$ which is assumed to be convex and compact. 
Let us first describe the convergence of the sequence of objective function values $\{\Phi(Z^{(i)}) \}$ obtained with  Algorithm~\ref{SCCAE-NMF}. Since $\Phi(Z)$ is concave, its linearization around the current iterate \ngi{$Z^{(i)}$}  provides an upper approximation, that is, 
\begin{equation}\label{opti_concavfun}
     \Phi(Z) \leq \Phi(Z^{(i)})+\langle \nabla \Phi(Z^{(i)}), Z-Z^{(i)}\rangle 
     \; \text{ for all } \; Z \in Q.
\end{equation}
This upper bound is tight at the current iterate and is exactly minimized over the feasible set $Q$ at each iteration. Hence Algorithm~\ref{SCCAE-NMF} is a majorization-minimization algorithm.
This implies that $\Phi(Z)$ is nonincreasing under the updates of Algorithm~\ref{SCCAE-NMF} and since $\Phi(Z)$ is bounded below on $Q$, by construction, the sequence of objective function values $\{ \Phi(Z^{(i)}) \}$ converges. 

We now focus on the convergence analysis of the sequence of iterates $\{ Z^{(i)} \}$ generated by Algorithm~\ref{SCCAE-NMF}, in particular convergence to a stationary point. To achieve this goal, we first recall some basics about the Frank-Wolfe (FW) algorithm. The FW algorithm \cite{RePEc:wly:navlog:v:3:y:1956:i:1-2:p:95-110} is a popular first-order method to solve \eqref{opt_prob} that relies on the ability to compute efficiently the so-called Linear Minimization Oracle (LMO), 
that is, $LMO(D):=\underset{Z \in Q}{\text{argmin }} \langle D , Z \rangle$ where $D$ denotes some search direction. The FW algorithm with adaptive step size is given in Algorithm~\ref{FW}. A step of this algorithm can be briefly summarized as follows: at a current iterate $Z^{(i)}$, the algorithm considers the first-order model of the objective function (its linearization), and moves towards a minimizer of this linear function, computed on the same domain $Q$.
\begin{algorithm}[H]
\begin{algorithmic}[1]
\STATE $Z^{(0)} \in Q$, number of iterations \textit{I}.
\FOR {$i=1,2,\dots,$ I}
    \STATE Compute $V^{(i)} := \underset{Z \in Q}{\text{argmin }} \langle \nabla \Phi(Z^{(i-1)}), Z\rangle$ 
    \STATE Choose $0 < \tau^{(i)} \leq 1$. (A standard choice in the literature is $\tau^{(i)}:=\frac{2}{i+1}$.)
    \STATE Update $Z^{(i)}:=(1-\tau^{(i)})Z^{(i-1)}+\tau^{(i)} V^{(i)}$
\ENDFOR
\end{algorithmic}
\caption{Frank-Wolfe algorithm \label{FW}}
\label{alg:seq}
\end{algorithm}

Algorithm~\ref{SCCAE-NMF} is a particular case of Algorithm~\ref{FW} for which $\tau^{(i)}=1$ for all $i$. 
In the last decade, FW-based methods have regained interest in many fields, mainly driven by their good scalability and the crucial property that Algorithm~\ref{FW} maintains its iterates as a convex combination of a few extreme points. 
This results in the generation of sparse and low-rank solutions since at most one additional extreme point of the set $Q$ is added to the convex combination at each iteration. More details and insights about the later observations can be found in \cite{10.1145/1824777.1824783,jaggi2011}. 
 FW algorithms have been recently studied in terms of convergence guarantees for the minimization of various classes of functions over convex sets, such as convex functions with Lipschitz continous 
gradient~\cite{pmlr-v28-jaggi13}, and non-convex differentiable 
functions~\cite{lacoste2016convergence}.
However, we are not aware of any convergence rates proven for Algorithm~\ref{FW} when solving \ngi{\eqref{opt_prob} assuming only concavity of the objective}. To derive  rates in the concave setting, we need to define a measure  of stationarity for our iterates. In this paper, we consider the so-called FW gap of $\Phi$ at $Z^{(i)}$ defined as follows:
\begin{equation}\label{mui_def}
    \mu^{(i)} := \underset{Z \in Q}{\max} \langle \nabla \Phi(Z^{(i)}), Z^{(i)}-Z\rangle . 
\end{equation} 
This quantity is standard in the analysis of FW algorithms, see \cite{pmlr-v28-jaggi13,lacoste2016convergence} and the references therein.
A point $Z^{(i)}$ is a stationary point for the constrained optimization problem \eqref{opt_prob} if and only if $  \mu^{(i)}=0$. Moreover, the FW gap 
\begin{itemize}
    \item \ngi{provides a lower bound on the accuracy:  
    $0 \leq \mu^{(i)} \leq \Phi(Z^{(i)}) - \Phi^*$ for all $i$, where $\Phi^*:= \min_{Z \in Q} \Phi(Z)$}, 
    \item is affine invariant, that is, it is invariant with respect to an affine transformation of the domain $Q$ in problem \eqref{opt_prob} \cite{pmlr-v28-jaggi13},  and 
    \item is not tied to any specific choice of norms, unlike criteria such as $\left\|\nabla \Phi(Z^{(i)}) \right\|$.  
\end{itemize}

Let us provide a convergence rate for the FW algorithm (Algorithm~\ref{FW}).   
\begin{theorem}\label{Theo1}
Consider the problem \eqref{opt_prob} where $\Phi$ is a continuously differentiable concave function over the compact convex domain $Q$. 
Let us denote $Z^{(i)}$ the sequence of iterates generated by the FW algorithm (Algorithm~\ref{FW}) applied on~\eqref{opt_prob}. \revise{Assume there exists a constant $\tilde{\tau} > 0$ such that $\tilde{\tau} \leq \tau^{(i)} \leq 1$ for all $i$.} 
Then the minimal FW gap, 
defined as $\tilde{\mu}^{(i)} := \min_{0 \leq j \leq i} \, \mu^{(j)}$, 
satisfies, for all $i=1,2,\dots$, 
\begin{equation}\label{convRate}
    \tilde{\mu}^{(i)} \; \leq \; \frac{1}{\tilde{\tau}} \frac{\Phi(Z^{(0)})-\Phi^*}{i+1},
\end{equation}
where  
 $\Phi(Z^{(0)})-\Phi^*$ is the (global) accuracy of the initial iterate. 
\end{theorem} 
\begin{proof}
\revise{Using~\eqref{opti_concavfun}}, 
 any points $Z^{(i)}$ and $Z^{(i+1)}$ generated by Algorithm~\ref{FW} satisfy    
\begin{equation}\label{conca_Phi}
    \Phi(Z^{(i+1)}) \leq \Phi(Z^{(i)}) + \langle \nabla \Phi(Z^{(i)}),Z^{(i+1)}-Z^{(i)} \rangle . 
\end{equation}
Let \revise{us} substitute  $Z^{(i+1)}$ by its construction in Algorithm~\ref{FW} (line 5) in \eqref{conca_Phi} to obtain 
\begin{equation}\label{demo_eq2}
    \begin{aligned}
       \Phi(Z^{(i+1)}) & \leq  \Phi(Z^{(i)}) + \langle \nabla \Phi(Z^{(i)}),(1-\tau^{(i)})Z^{(i)}+\tau^{(i)} V^{(i)}-Z^{(i)} \rangle \\
         & = \Phi(Z^{(i)}) - \tau^{(i)} \langle \nabla \Phi(Z^{(i)}),Z^{(i)}- V^{(i)} \rangle . 
    \end{aligned}
\end{equation} 
Let us show that $\langle \nabla \Phi(Z^{(i)}),Z^{(i)}- V^{(i)} \rangle$ is equal to $\mu^{(i)}$ defined in Equation~\eqref{mui_def}: 
\begin{equation}
    \begin{aligned}
        \mu^{(i)} & = \underset{Z \in Q}{\max} \langle \nabla \Phi(Z^{(i)}), Z^{(i)}-Z\rangle \\
              & = \langle \nabla \Phi(Z^{(i)}), Z^{(i)}\rangle + \underset{Z \in Q}{\max} \langle \nabla \Phi(Z^{(i)}), -Z\rangle \\
              & = \langle \nabla \Phi(Z^{(i)}), Z^{(i)}\rangle - \underset{Z \in Q}{\min} \langle \nabla \Phi(Z^{(i)}), Z\rangle \\
              & = \langle \nabla \Phi(Z^{(i)}), Z^{(i)}\rangle - \langle \nabla \Phi(Z^{(i)}), V^{(i)}\rangle = \langle \nabla \Phi(Z^{(i)}), Z^{(i)}-V^{(i)}\rangle, 
    \end{aligned}
\end{equation}
where the fourth equality holds be definition of $V^{(i)}$ from Algorithm~\ref{FW} (line 3). Equation~\eqref{demo_eq2} becomes:
\begin{equation}\label{eq3}
    \Phi(Z^{(i+1)}) \leq \Phi(Z^{(i)}) - \tau^{(i)} \mu^{(i)} . 
\end{equation}
By recursively applying \eqref{eq3} for the iterates generated by Algorithm~\ref{FW}, we obtain 
\begin{equation}\label{eq4}
    \Phi(Z^{(i+1)}) \leq \Phi(Z^{(0)}) - \sum_{j=0}^{i} \tau^{(j)} \mu^{(j)} . 
\end{equation}
Let define the quantities $\tilde{\mu}^{(i)}:=\underset{0 \leq j \leq i}{\min}  \mu^{(j)} $, the minimal FW gap encountered along iterates, and $\tilde{\tau}$ such that $\tilde{\tau} \leq \tau^{(i)} \leq 1$ for all $i \geq 0$, so that inequality~\eqref{eq4} implies  
\begin{equation}\label{eq5}
\begin{aligned}
    & \Phi(Z^{(i+1)}) \leq \Phi(Z^{(0)}) - (i+1) \tilde{\tau} \tilde{\mu}^{(i)} 
    \iff \tilde{\mu}^{(i)} \leq \frac{1}{\tilde{\tau}} \frac{\Phi(Z^{(0)})-\Phi(Z^{(i+1)})}{i+1} . 
\end{aligned}
\end{equation}
Finally, using the fact that $\Phi(Z^{(0)})-\Phi(Z^{(i+1)}) \leq \Phi(Z^{(0)}) - \Phi^*$ where $\Phi^*:= \underset{Z \in Q}{\min} \Phi(Z)$, 
inequality~\eqref{eq5} becomes 
\begin{equation} \label{convaceupperbound}
     \tilde{\mu}^{(i)} \leq \frac{1}{\tilde{\tau}} \frac{\Phi(Z^{(0)})-\Phi^*}{i+1}, 
\end{equation}
which concludes the proof. 
\end{proof}

Theorem~\ref{Theo1} shows that it takes at most $\mathcal{O}(\frac{1}{\epsilon})$ iterations to find an approximate stationary point with gap smaller than $\epsilon$. 
\revise{Note that Theorem~\ref{Theo1}  requires $\min_i \tau^{(i)} > 0$ and, for example, the standard choice  $\tau^{(i)} = \frac{2}{i+1}$ does not satisfy this assumption.} 
}

\revise{
\subsection{Compactness assumption and convergence of Algorithm~\ref{SCCAE-NMF}}  \label{rem:compact} 

By looking at the convergence rate given by~\eqref{convRate}, it is tempting to take $\tilde{\tau}$ as large as possible. However, since the set $Q$ is convex, the maximum allowed value for $\tilde{\tau}$ is 1 to ensure that the iterates $Z^{(i)}$ remain feasible. \revise{This setting leads to a convergence rate of $\mathcal{O}(1/i)$, given the assumptions of Theorem~\ref{Theo1} are satisfied,} and corresponds to Algorithm~\ref{SCCAE-NMF}. 
In Theorem~\ref{Theo1}, we need the set $Q$ to be compact. 
Let us discuss this assumption in the context of Algorithm~\ref{SCCAE-NMF} that relies on our two formulations: 
\begin{enumerate}
    \item For the formulation using exponential cones~\eqref{optiprob2}, a variable $W_{fk}$ (resp.\ $H_{kn}$) equal to zero will correspond to $U_{fk}$ (resp.\ $T_{kn}$) going to $-\infty$, which is not bounded. 
    As recommended by Mosek, we use additional artificial component-wise lower and upper bounds for $U$ and $T$, namely, -35 and 10. Therefore, our implementation actually solves a component-wise bounded version of~\eqref{optiprob2}. However, 
    since $e^{-35} \approx 6 \cdot 10^{-16}$ and 
    $e^{10} \approx 2 \cdot 10^{4}$, numerically, 
    these constraints do not exclude good approximations of $V$ in practice,  as long as the entries in $V$ belong to a reasonable range which can be assumed w.l.o.g.\ by a proper preprocessing of $V$, e.g., $V \leftarrow \frac{V}{\max_{f,n} V_{fn}}$ so that $V_{fn} \in [0,1]$ for all $f,n$; see, e.g., the discussion in~\cite[page 66]{gillis2011nonnegative}.

    \item For the formulation using second-order  cones~\eqref{optiprob4}, we can assume compactness w.l.o.g.  
    In fact, for simplicity, let us consider the formulation~\eqref{basicNMFoptiprob2} in variables $(W,H)$, since the change of variables is the component-wise square root, 
    and keeps the feasible set compact. 
    We can w.l.o.g.\ add a set of normalization constraints on the rows of $H$, such as  $\sum_n  H_{kn} = \| H_{k:} \|_1 = 1$ for all $k$ which leads to a compact set for $H$, since we can use the degree of freedom in scaling between the columns of $W$ and rows of $H$. It remains to show that $W$ can be assumed to be in a compact set. 
    Let us show that the level sets of the objective function are compact, which will give the result. 
    The objective function of~\eqref{basicNMFoptiprob2} is the component-wise $\ell_1$ norm, $\|WH\|_1$. 
    Then, let 
    $(W',H')$ be an arbitrary feasible solution of~\eqref{basicNMFoptiprob2} (such a solution can be easily constructed), and add the constraint $\|WH\|_1 \leq \|W'H'\|_1 = f'$ to formulation~\eqref{basicNMFoptiprob2}, which does not modify its optimal solution set.     
    We have for all $k$
    \[
     \| W_{:k} \|_1  
     = 
     \| W_{:k} \|_1 \| H_{k:} \|_1 
    =  
      \| W_{:k} H_{k:} \|_1 
       \leq 
    \| W H \|_1 \leq f', 
    \] 
    which shows that $W$ in that modified formulation also belongs to a compact set.     The second equality and first inequality follow from nonnegativity of $W$ and $H$. 
\end{enumerate}

under these modifications that make the feasible set $Q$ compact, 
we have the following corollary. 
\begin{corollary} 
Both variants~\eqref{optiprob2} and~\eqref{optiprob4} of Algorithm~\ref{SCCAE-NMF} generate a sequence of iterates $Z^{(i)}$  whose FW gap converges according to $\tilde{\mu}^{(i)} \leq \frac{\Phi(Z^{(0)})-\Phi^*}{i+1}$. 
\end{corollary}  
}

\begin{remark}[Difference-of-convex-functions optimization] 
Algorithm~\ref{SCCAE-NMF} can also be interpreted as a special case of a difference-of-convex algorithm that minimizes the difference between two convex functions, namely $f_1(Z)-f_2(Z)$~\cite{le2018dc}. In our case, 
we would have $f_1(Z) = 0$. For such problems, a convergence rate to stationary points of order $O(1/i)$ has also been derived when optimizing over a convex compact feasible set, but using a different measure of stationary~\cite{abbaszadehpeivasti2021rate}.
\end{remark}

\revise{
\begin{remark}
Using a proper normalization (see Section~\ref{rem:compact}), the original NMF problem can be tackled directly by the FW algorithm (Algorithm~\ref{FW}), as the linear minimization oracle can be computed in closed form. 
Hence one could use existing results on FW to derive a convergence 
rate. However, the FW algorithm applied to smooth nonconvex problems leads to a worse rate of $O(1/\sqrt{i})$~\cite{lacoste2016convergence}. 
\end{remark}
}

\section{Numerical experiments}\label{num_exp}


In this section, Algorithm~\ref{SCCAE-NMF}, using both formulations \eqref{optiprob2} and \eqref{optiprob4}, is tested for the computation of exact NMF for particular classes of matrices usually considered in the  literature: 
(1)  $10$-by-$10$ matrices randomly generated with nonnegative rank $5$ (each matrix is generated by multiplying two random rank-5 nonnegative matrices), 
(2) four $6$-by-$6$ infinitesimally rigid factorizations with nonnegative rank $5$~\cite{krone2020uniqueness}, denoted $V_{\text{inf}i}$ for $i=1,2,3,4$, 
and 
(3) four $5$-by-$5$ slack matrices corresponding to nested hexagons, denoted $V_{a=x}$,  with nonnegative ranks $3,4,5,5$ depending on a parameter $x = 2,3,4,+\infty$, respectively. 
These matrices are described in more details in the Appendix~\ref{appA}. 

\revise{
In order to make Algorithm~\ref{SCCAE-NMF} practically more effective, we incorporate two improvements in the context of Exact NMF: 
\begin{enumerate}
    \item Sparsity patterns integration: in NMF, entries of $W$ and $H$ are expected to be equal to zero at optimality. Hence, 
    when some entries of $W$ or $H$ are sufficiently close to zero, fixing them to zero for all remaining iterations reduces the number of variables and hence accelerates the subsequent iterations of Algorithm~\ref{SCCAE-NMF}; see Appendix~\ref{appB} for the details. 
    
    \item Final refinement:  
    once a solution is generated by Algorithm~\ref{SCCAE-NMF} for the formulation based on exponential cones, it  can  typically be slightly improved by applying a standard NMF algorithm (we use a few iterations of A-HALS~\cite{6797154}). 
    This is due to the bound constraints on $W$ and $H$; see Section~\ref{rem:compact}. 
\end{enumerate}
}



For each of the matrices, we run our Algorithm~\ref{SCCAE-NMF} for 750 iterations for nested hexagons and random matrices and 3000 iterations for infinitesimally rigid matrices,  
and compare with the state-of-the-art algorithms from~\cite{vandaele2016} with Multi-Start 1 heuristic "ms1" and the Rank-by-rank heuristic "rbr". For each method, we run 100 initializations with SPARSE10, as recommended in~\cite{vandaele2016}, with target precision $10^{-6}$. Note that a different random matrix is generated each time for the experiments on random matrices, 
following the procedure described in the Appendix~\ref{appA}.  \vlepi{All tests are preformed using Matlab R2021b on a laptop
Intel CORE i7-11800H CPU @2.3GHz 16GB RAM. The code is available from \url{https://bit.ly/3FqMqhD}.}

Table~\ref{tab:quali_measu_SCCAENMF} reports the number of successes over 100 attempts to compute the exact NMF of the input matrices, where the success is 
defined as obtaining a solution where $\frac{\left\|V-WH \right\|_F}{\left\| V \right\|_F}$ is below the target precision, namely $10^{-6}$.

\begin{table}[ht!]
\begin{center}
\caption{Comparison of the two variants of Algorithm~\ref{SCCAE-NMF} for \eqref{optiprob2} and ~\eqref{optiprob4} with the algorithm from~\cite{vandaele2016} with the "ms1" and "rbr" heuristics. Each run is performed with 100 initializations to compute the factorizations of matrices described in Appendix~\ref{appA}. In bold, we indicate the algorithm that found the most exact NMFs.}   
\label{tab:quali_measu_SCCAENMF}
\resizebox{\columnwidth}{!}{%
\begin{tabular}{|c|c|c|c|c|c|} 
\hline
\multicolumn{2}{|c|}{ }      & Algorithm~\ref{SCCAE-NMF}   & Algorithm~\ref{SCCAE-NMF}   & Algorithm from~\cite{vandaele2016} & Algorithm from~\cite{vandaele2016}\\
\multicolumn{2}{|c|}{ }      &  for \eqref{optiprob2} &  for \eqref{optiprob4}  & with ``ms1'' & with ``rbr''\\
\hline
\multicolumn{2}{|c|}{\textbf{Matrices}}& $\textbf{/100}$ & $\textbf{/100}$ & $\textbf{/100}$ & $\textbf{/100}$\\
\hline
\multicolumn{2}{|c|}{Random matrices} & 100 & 100 & 100  & 100\\
\hline  
Inf. Rig. Fac. & ${V_{\text{inf}1}}$ &5 & \textbf{7} & \vlepi{6} & 0\\
 &  $V_{\text{inf}2}$ &\vlepi{16} & \vlepi{38}& 34&\textbf{97}\\
 & $V_{\text{inf}3}$ &14 & \vlepi{33} &14& \textbf{90}\\
 &  ${V_{\text{inf}4}}$ &\vlepi{16} & \textbf{\vlepi{20}}& 15& 0\\ \hline  
 
Nested hexagons & $V_{a=2}$ &100 &100 &100 &100 \\ 
    & $V_{a=3}$ &100 &100 & 100 &100 \\ 
    & $V_{a=4}$ &\vlepi{35}&\vlepi{69} & 36 & \textbf{100}\\ 
    & $V_{a \rightarrow +\infty}$ &17 &\vlepi{42} & 20 & \textbf{100} \\ \hline
\end{tabular}
}
\end{center}
\end{table}

We observe the following: 
\begin{itemize}

\item All algorithms find exact NMFs in all runs for random matrices. It is well-known that factorizing randomly generated matrices is typically easier~\cite{vandaele2016}. 
This shows that Algorithm~\ref{SCCAE-NMF} with both formulations \eqref{optiprob2} and \eqref{optiprob4} is also able to compute exact NMFs in this scenario, which is reassuring. 

    \item Looking at the nonrandom matrices, Algorithm~\ref{SCCAE-NMF} with both formulations \eqref{optiprob2} and \eqref{optiprob4}, and ``ms1'' from~\cite{vandaele2016} are the only algorithms able to compute an exact NMF for at least some of the 100 initializations. Moreover, among these three algorithms,  
    Algorithm~\ref{SCCAE-NMF} with \eqref{optiprob4} found most frequently an exact NMF.
    
    \item For some matrices (namely, $V_{a=4}$ and $V_{a \rightarrow +\infty}$), ``rbr'' from~\cite{vandaele2016} is able to compute exact NMF for all initializations, which is not the case of the other algorithms. 
    However, ``rbr'' is not able to compute exact NMFs for $V_{\text{inf}1}$ and  $V_{\text{inf}2}$.


\end{itemize}

In summary, Algorithm~\ref{SCCAE-NMF} competes favorably with the algorithms proposed in~\cite{vandaele2016}, and appears to be more robust in the sense that it computes exact NMFs in all the tested cases. 
In addition, the second-order cone formulation, 
Algorithm~\ref{SCCAE-NMF} with \eqref{optiprob4}, 
performs slightly better than with the exponential cone formulation, 
Algorithm~\ref{SCCAE-NMF} with \eqref{optiprob2}. 


\paragraph*{Computational time}
In terms of computational time, Algorithm~\ref{SCCAE-NMF} \vlepi{performs similarly to} algorithm from~\cite{vandaele2016} \vlepi{for the considered matrices, but it does not scale as well}. The main reason is that it relies on interior-point methods, while \cite{vandaele2016} relies on first-order methods (more precisely, exact BCD). For example, for the infinitesimally rigid matrices,  
\vlepi{Algorithm~\ref{SCCAE-NMF} and \cite{vandaele2016} take between 2 and 16 seconds.
We would report slower performance for Algorithm~\ref{SCCAE-NMF} compared to \cite{vandaele2016} for larger matrices.} Hence a possible direction of research would be to use faster methods to tackle the  
conic optimization problems.  

\vlepi{
\paragraph*{Numerical validation of the rate of convergence}
As a simple empirical validation of the rate of convergence proposed in Section~\ref{sec_convRes}, we report in Figure~\ref{fig:emprival} the evolution of the minimum FW gap computed along iterations by Algorithm~\ref{SCCAE-NMF} (for \eqref{optiprob2} and \eqref{optiprob4}) for one tested matrix, namely ${V_{\text{inf}2}}$. Similar behaviours were observed for all the tested input matrices: 
additional figures are given in Appendix~\ref{appC}.  
Note that we also integrated a variant of Algorithm~\ref{SCCAE-NMF} for which  $\tau^{(i)}:=\frac{2}{i+1}$ (a standard choice in the literature for FW algorithms). 
In Figure~\ref{fig:emprival}, we observe that the behaviour of the minimal FW gap is in line with the theoretical results from Section~\ref{sec_convRes}.
Furthermore, the choice $\tilde{\tau}=\tau^{(i)}:=1$ leads to faster decrease of the minimal FW gap encountered along iterations, as expected. 
}


\begin{figure}[!tbp]
  \centering
  \subfloat[Algorithm~\ref{SCCAE-NMF} for \eqref{optiprob2}]{\includegraphics[width=0.49\textwidth]{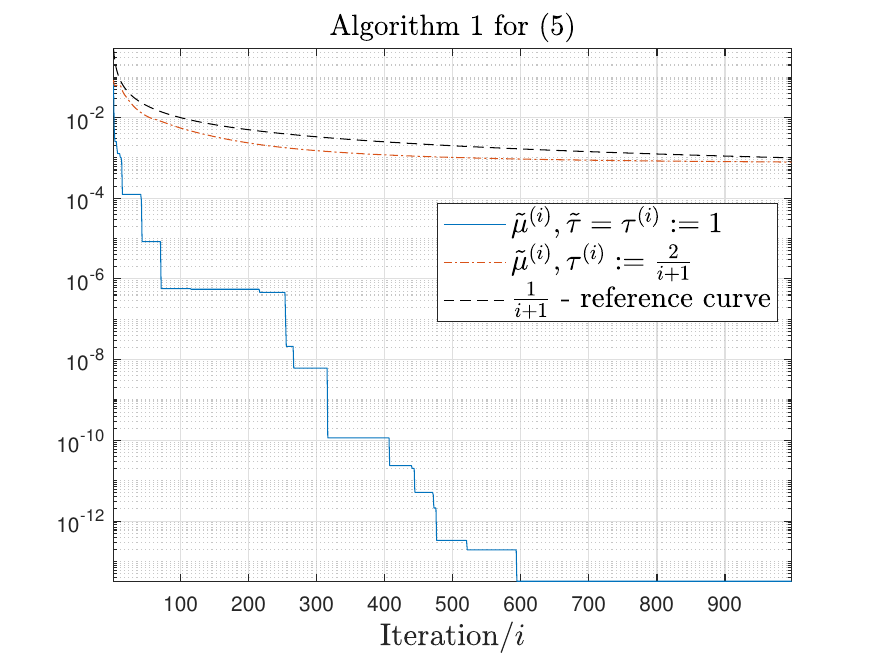}\label{fig:f1}}
  \hfill
  \subfloat[Algorithm~\ref{SCCAE-NMF} for \eqref{optiprob4}]{\includegraphics[width=0.49\textwidth]{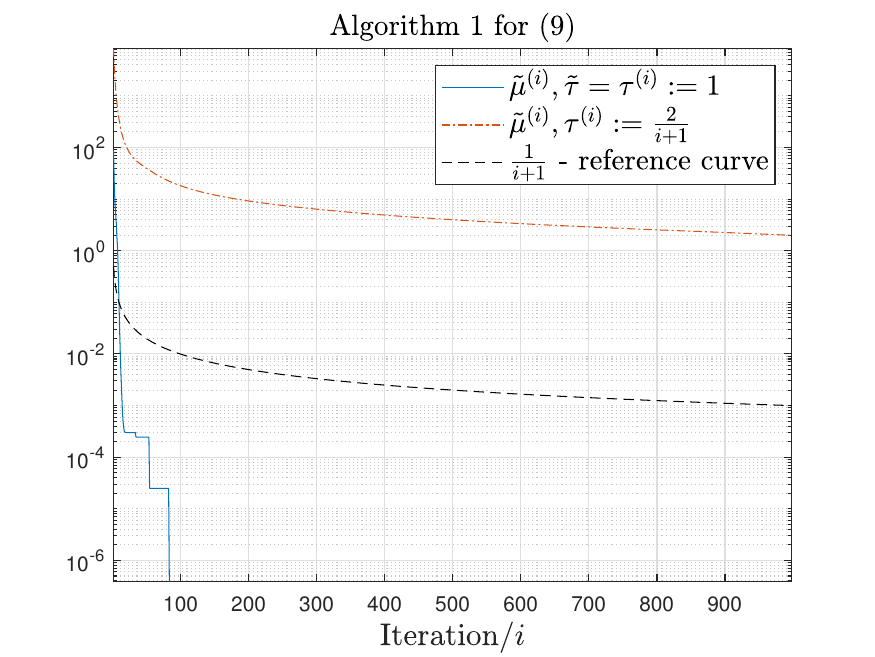}\label{fig:f2}}
  \caption{Evolution of the minimum FW gap computed along iterations by Algorithm~\ref{SCCAE-NMF} for \eqref{optiprob2} (left) and \eqref{optiprob4} (right) for the matrix ${V_{\text{inf}2}}$.}\label{fig:emprival}
\end{figure}


\paragraph*{Impact of the initialization} 

We also investigated the impact of using an improved initialization for our Algorithm~\ref{SCCAE-NMF} with \eqref{optiprob4}, based on a rank-one NMO slightly perturbed with nonnegative random uniform values. More precisely, we set $Z^{(0)} = Z_{K=1} + d R \frac{\|Z_{K=1}\|_F}{\|R\|_F}$ where $Z_{K=1}$ is the rank-one NMO computed with the methodology detailed in 
Section~\ref{subsec_rank1NMO}, $R$ is a matrix of appropriate size of uniformly distributed random numbers in the interval $(0,1)$, and $d$ is a parameter to be defined by the user. 
In our numerical experiments, we chose values for $d$ within the interval $[0.01, 0.05]$. We find that Algorithm~\ref{SCCAE-NMF} with \eqref{optiprob4} using that dedicated initialization found respectively 74, 65 and 52 successes for matrices $V_{a \rightarrow +\infty}$, $V_{\text{inf}2}$ and $V_{\text{inf}4}$, respectively, a marked improvement over the default random initialization (where it had 42, 38, 20 successes, respectively). For the other tested matrices, it did not change the result significantly (only a few additional successes). Therefore a possible direction of research would be the design of more advanced strategies for the initialization of $(U, T)$. 


\section{Conclusion} 

In this paper, we introduced two formulations for computing exact NMFs,  namely \eqref{basicNMFoptiprob}  and \eqref{basicNMFoptiprob2} that are  under- and upper-approximation formulations for NMF, respectively. 
For each formulation, we used a particular change of variables that enabled the use of two convex cones, namely the exponential and second-order cones, to reformulate these problems as the minimization of a concave function over a convex feasible set.   
In order to solve the two optimization problems, we proposed   Algorithm~\ref{SCCAE-NMF} that relies on the resolution of successive linear  approximations of the objective functions and the use of interior-point methods.  \vlepi{We also showed that our optimization scheme relying on successive linearizations is a special case of the Frank-Wolfe (FW) algorithm. 
Using an appropriate measure of stationarity, namely the FW gap, 
we showed in Theorem~\ref{Theo1} that the minimal FW gap generated by our algorithm  converges as $\mathcal{O}(\frac{1}{i})$, where $i$ is the iteration index.} 
\ngi{We believe this type of global convergence rate to a stationary point is new for NMF.} 
We showed that Algorithm~\ref{SCCAE-NMF} is able to compute exact NMFs for several classes of nonnegative matrices (namely, randomly generated matrices, infinitesimally  rigid matrices, and slack matrices of nested hexagons) and as such demonstrate its competitiveness compared to recent methods from the literature. However, we have only tested  Algorithm~\ref{SCCAE-NMF} on a limited number of nonnegative matrices \revise{for exact NMF}. 
In the future we plan to test it on a larger library of nonnegative matrices \revise{and also to compute approximate NMFs in data analysis applications}, in order to better understand the behavior of Algorithm~\ref{SCCAE-NMF} along with the two formulations, \eqref{optiprob2} and \eqref{optiprob4}. 

Further works also include:
\begin{itemize}
    
    \item The design of more advanced strategies for the initialization of $(U,T)$.
    
    \item The elaboration of alternative formulations for \eqref{optiprob2} and \eqref{optiprob4} to deal with the non-uniqueness of the NMF models. In particular, we plan to develop new formulations that discard solutions of the form  $V=\tilde{W}\tilde{H}=\left(WE\right)\left(E^{-1}H\right)$ for a given solution $(W,H)$ and for invertible matrices $E$ such that $WE\geq 0$ and $E^{-1}H\geq 0$. 
    For example, one could remove the permutation and scaling ambiguity for the columns of $W$ and rows of $H$, to remove some degrees of freedom in the formulations. 
    We refer the interested reader to~\cite{fu2019nonnegative} and \cite[Chapter 4]{gillis2020bk}, and the references therein, for more information on the non-uniqueness of NMF. 
    
    \item The use of our framework for other closely related problems; in particular the computation of symmetric NMFs which requires $H = W^\top$; this problem is closely related to completely positive matrices~\cite{berman2003completely}. Symmetric NMF can be used for data analysis and in particular for various clustering tasks~\cite{doi:10.1137/1.9781611972825.10}.
    
    \item \vlepi{The use of upper-approximations that are more accurate     
    than linearizations for concave functions such as second-order models, and  study the convergence for such models.}
    
\end{itemize}

\revise{
 \section*{Acknowledgements}

We thank the two reviewers for their careful reading and insightful feedback that helped us improve the paper significantly. 
} 












\bibliographystyle{tfs}
\bibliography{interacttfssample}

\newpage 

\appendix 

\section{Factorized matrices} \label{appA}

In this appendix, we describe the matrices considered for the numerical experiments  in Section~\ref{num_exp}:

\begin{itemize}
    \item Randomly generated matrices: It is standard in the NMF literature to use randomly generated matrices to compare algorithms (see, e.g.,~\cite{doi:10.1137/110821172}). 
    In this paper, we used the standard approach: $V = WH \in \mathbb{R}_+^{F \times N}$ where each entry of $W \in \mathbb{R}_+^{F \times K}$ and $H \in \mathbb{R}_+^{K \times N}$ is generated using the uniform distribution in the interval $[0,1]$, and $K \leq \min(F,N)$.  
    With this approach, $\rank(V) = \rank_+(V) = K$ with probability one. 
    In the numerical experiments reported in Section~\ref{num_exp}, we used $F=N=10$ and $K=5$.
    
    \item Infinitesimally rigid factorizations: in this paper, we consider four infinitesimally rigid factorizations for $5 \times 5$ matrices with positive entries and with nonnegative rank equal to four from~\cite{krone2020uniqueness}:
    \begin{equation*} 
     \begin{aligned}
        & V_{\text{inf}1} = \left( 
        \begin{array}{ccccc}
            573705 &    806520 &    167622 &    246500 & 531659 \\
            397096 &    39600 &    299176 &    63720 & 274120  \\
            131646 &    403260 &    30269 &    226915 & 264510  \\
            9114 &    85160 &    311182 &    827468 & 851798  \\
            147857 &    3200 &    351037 &    599025  & 697755     
        \end{array} 
        \right), \\
        & V_{\text{inf}2} = \left( 
        \begin{array}{ccccc}
            30893 &    319912 &    149770 &    873& 111428 \\
            383490 &    87990 &    5580 &    628440 & 587250  \\
            560076 &    1030324 &    331070 &    288045 & 350647  \\
            203830 &    305184 &    277512 &    264376 & 205933  \\
            90911 &    142936 &    500784 &    618842  & 609633
             
        \end{array} 
        \right), \\
        & V_{\text{inf}3} = \left( 
        \begin{array}{ccccc}
            948201 &    723609 &    958755 &    591858& 397953 \\
            222448 &    218040 &    30429 &    348793 & 15825  \\
            329588 &    7189 &    623001 &    12012 & 469185  \\
            467424 &    160704 &    115092 &    835504 & 343912  \\
            1114797 &    932972 &    975775 &    997164  & 636096
             
        \end{array} 
        \right),\\
        & V_{\text{inf}4} = \left( 
        \begin{array}{ccccc}
            88076 &    294646 &    658787 &    902872& 244559 \\
            2216 &    4216 &    596705 &    652698 & 250465  \\
            279360 &    180864 &    769506 &    1051380 & 391634  \\
            553284 &    826606 &    765406 &    293965 & 883775  \\
            696039 &    897917 &    148301 &    832169  & 169525
             
        \end{array} 
        \right).
        \end{aligned}
    \end{equation*}
    These matrices have been shown to be challenging to factorize. We refer the reader to~\cite{krone2020uniqueness} for more details. 
    \item Nested hexagons problem: 
     computing an exact NMF is equivalent to tackle a 
    well-known problem in computational geometry which is referred to as nested polytope problem. Here we consider a family of input matrices whose exact NMF  corresponds to finding a polytope nested between two hexagons; see \cite[Chapter 2]{gillis2020bk} and the references therein.  
    Given $x>1$, $V_{a=x}$ is defined as 
    \begin{equation*} 
        \frac{1}{x} \left( 
        \begin{array}{cccccc}
            1 &    x &    2x-1 &    2x-1 & x &  1 \\
            1 &    1 &    x &    2x-1 & 2x-1 &  x \\
            x &    1 &    1 &    x & 2x-1 &  2x-1 \\
            2x-1 &    x &    1 &    1 & x &  2x-1 \\
            2x-1 &    2x-1 &    x &    1 & 1 &  x \\
            x &    2x-1 &    2x-1 &    x & 1 &  1 
        \end{array} 
        \right) 
    \end{equation*}
    which satisfies $\rank(V_{a=x})=3$ for any $x > 1$.  
    The inner hexagon is smaller than the outer one with a ratio of $\frac{a-1}{a}$. We consider three values for $a$:
    \begin{itemize}
        \item $a=2$: the inner hexagon is twiced smaller than the outer one and we can fit a triangle between the two, thus $\rank_+(V_a)=3$.
        \item $a=3$: the inner hexagon is $2/3$ smaller than the outer one and we can fit a rectangle between the two, thus $\rank_+(V_a)=4$.
        \item $a=4$: $\rank_+(V_a)=5$.
        \item $a \rightarrow +\infty$, which gives:
        \begin{equation*} 
        V = \left( 
        \begin{array}{cccccc}
            0 &    1 &    2&    2 &  1 & 0\\
            0 &    0 &    1 &    2& 2& 1  \\
            1 &    0 &    0 &    1 & 2 &  2 \\
            2 &    1 &    0 &    0& 1 &  2\\
            2 &    2 &    1&    0& 0 &  1 \\
            1 &    2 &    2 &    1 & 0 &  0 
             
        \end{array} 
        \right) 
    \end{equation*}
        with $\rank_+(V)=5$.
    \end{itemize}
\end{itemize}





\section{Sparsity Patterns Integration} \label{appB}

This appendix details the SPI procedure for quadratic cones, similar rationale has been followed for exponential cones.
Due to nonnegative constraints on the entries of $W$ and $H$, one can expect sparsity patterns for the solutions $(W,H)$, as for the solution $(U,T)$ of \eqref{optiprob4} since $W_{fk}=G(U_{fk})=\sqrt{U_{fk}}$ and $H_{kn}=G(T_{kn})=\sqrt{T_{kn}}$. Obviously, the sparsity for the solutions is reinforced by the sparsity of the input matrix $V$.
One can observe that the objective function $\Phi$ from \eqref{optiprob4} is not L-smooth on the interior of the domain, that is the non-negative orthant. In the case the $(f,k)$-entry of the current iterate $U^{i-1}$ tends to zero, the corresponding entry of the gradient of $\Phi$ w.r.t. $U$ tends to $\infty$ which therefore ends the optimization process. 
In order to tackle this issue and enables the solution to reach the desired tolerance of $10^{-6}$, we integrated an additional stage within the second building block of Algorithm~\ref{SCCAE-NMF}. This additional stage is referred to as "Sparsity Pattern Integration".
 Let us illustrate its principle on the following case: the entry $U_{\bar{f},\bar{k}}^{i-1}$ tends to zero. Let us now fix $U_{\bar{f},\bar{k}}$ to zero, drop this variable from the optimization process and observe the impact on the constraints of \eqref{optiprob3} in which variable $U_{\bar{f},\bar{k}}$ is involved;
 the inequality constraints identified by index $f=\bar{f}$ are:
     \begin{equation*}
         \sqrt{U_{\bar{f},1}} \sqrt{T_{1,n}} +...+ \sqrt{U_{\bar{f},\bar{k}}} \sqrt{T_{\bar{k},n}}+...+\sqrt{U_{\bar{f},K}} \sqrt{T_{K,n}} \geq  V_{\bar{f},n} \text{ for } n \in \mathcal{N}.
     \end{equation*}
First, since $\sqrt{U_{\bar{f},\bar{k}}}=0$, there is no more constraints on $\sqrt{T_{\bar{k},n}}$ for $N$ inequalities identified by index $f=\bar{f}$. Second, for the problem \eqref{optiprob4} and its successive linearizations, it is then clear that $N$ conic variables $t_{\bar{f},\bar{k},n}$ (and hence the $N$ associated conic constraints) can be dropped from the optimization process. 
Finally, 
the linear term $\left[ \nabla_{U} \Phi \left(U^{i-1},T^{i-1} \right) \right]_{\bar{f},\bar{k}} U_{\bar{f},\bar{k}}$ is removed from the current linearizations of $\Phi$. The same rationale is followed for the case entries of the current iterate for $T$ tend to zero. \\

On a practical point of view, at each activation of SPI, Algorithm~\ref{SCCAE-NMF} checks if entries of the current iterates $(U^{i-1},T^{i-1})$ are below a threshold $th$ defined by the user. Hence the corresponding entries of $U$ and $T$ are set to zero so that a sparsity pattern is determined, that are the indices of these entries. The successive linearizations of  \eqref{optiprob4} are automatically updated based on the current sparsity pattern with the approach explained above. \\

Let us illustrate the impact of triggering the SPI procedure on the solutions obtained for the factorization of the following matrix $V$:
\begin{equation} \label{example}
V = \left( 
\begin{array}{cccccc}
    0 &    1 &    2 &    2 & 1 & 0 \\
    0 &    0 &    1 &    2 & 2 & 1 \\
     1 &    0 &    0 &   1 & 2 & 2 \\
    2 &    1 &    0 &    0 &  1 & 2 \\
    2 &    2 &    1 &    0 & 0 &  1\\
    1 &    2 &    2 &    1 & 0 &  0
\end{array} 
\right).
\end{equation}
The nonnegative rank of \eqref{example} is known and is equal to $5$. Algorithm~\ref{SCCAE-NMF} is used to compute an exact NMF of $V$ with the following input parameters:
\begin{itemize}
    \item $K=5=\rank_+(V)$,
    \item $th=10^{-3}$,
    \item the maximum number of iterations defined by parameter \textit{maxiter} is set to 500 and the SPI procedure is triggered at iteration $400$.
\end{itemize}
Figure~\ref{fig:SPIdemonstration} displays the evolution of $\frac{\left\|V-WH \right\|_F}{\left\| V \right\|_F}$ along iterations for $V$ \eqref{example} with a factorization rank $K=5$. One can observe that, once the SPI is activated, the relative Frobenius error drops from 5 $10^{-4}$ to 8 $10^{-9}$, hence below the tolerance of $10^{-6}$ such that we assume an exact NMF $(W,H)$ is found. For this experiment, we obtain:

\begin{equation*} 
\begin{aligned}
& W = \left( 
\begin{array}{ccccc}
    0 &    1.4748 &    0.9259 &    0 & 0  \\
    0.7824 &    0 &    1.8517 &    0 & 0  \\
     0 &    0 &    0.9259 &   0 & 1.4716  \\
    0 &    0 &    0 &    0.6024 &  1.4716  \\
    0.7824 &    0 &    0 &    1.2049 & 0 \\
    0 &    1.4748 &    0 &    0.6024 & 0 
\end{array} 
\right), \\
& H = \left( 
\begin{array}{cccccc}
    0 &    0 &    1.2781 &    0 & 0 & 1.2781 \\
    0 &    0.6780 &    1.3561 &    0.6780 & 0  & 0\\
     0 &    0 &    0 &  1.0801 & 1.0801  & 0\\
    1.6599 &    1.6599 &    0 &    0 &  0 & 0  \\
    0.6796 &    0 &    0 &    0 & 0.6796  & 1.3591
\end{array} 
\right).
\end{aligned}
\end{equation*}

\begin{figure}[h!]
      \centering
	  \includegraphics[width=0.7\linewidth]{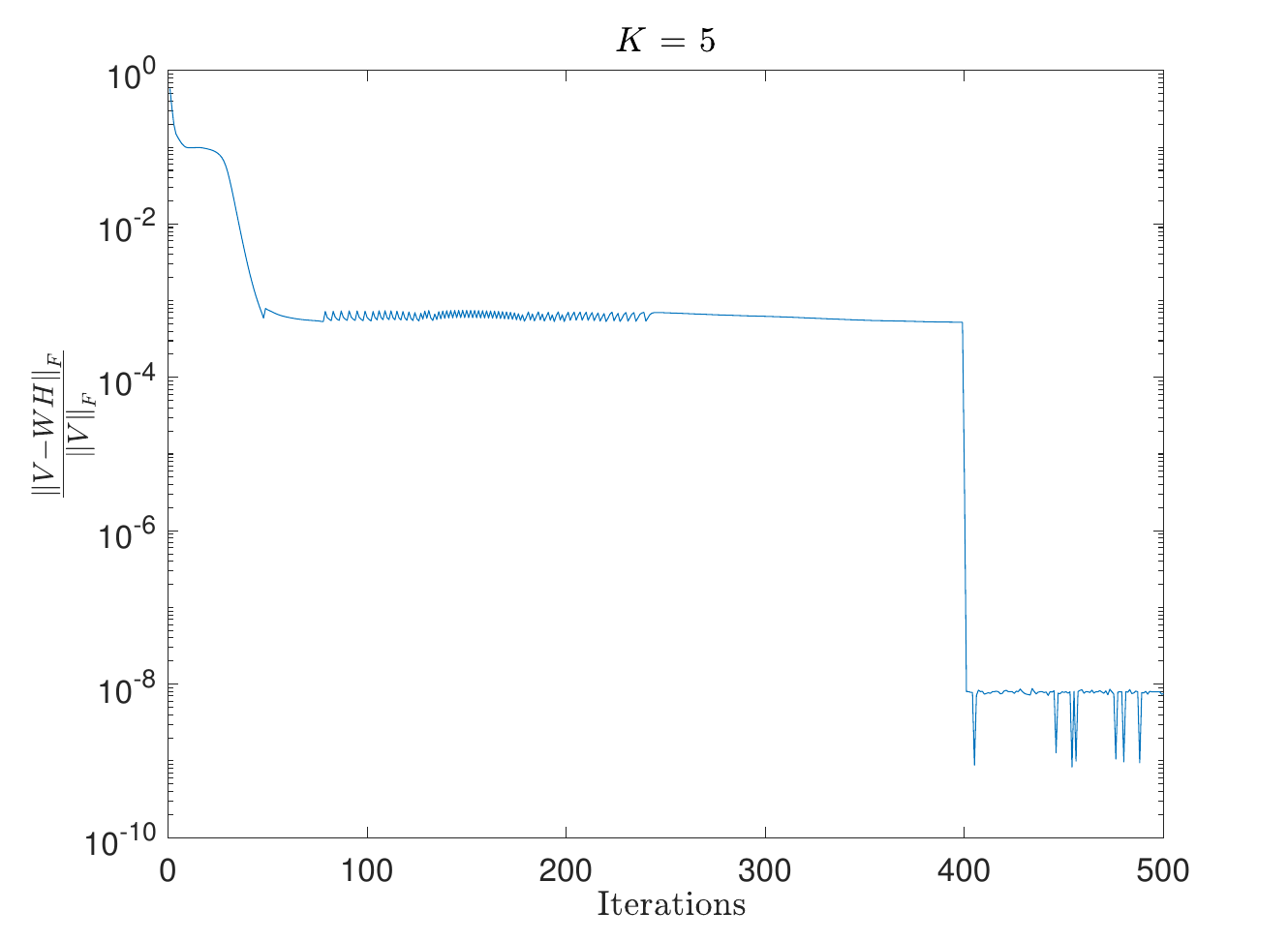}
	  \caption{Evolution of $\frac{\left\|V-WH \right\|_F}{\left\| V \right\|_F}$ along iterations; SPI is triggered at iteration $400$.}\label{fig:SPIdemonstration}
\end{figure}

\section{Additional empirical validations of the convergence rates} \label{appC} 

In this appendix, we report in Figures~\ref{fig:emprivalappendix1} and~\ref{fig:emprivalappendix2} additional empirical validations of the convergence rate introduced in~\ref{sec_convRes} for the two others classes on tested matrices, namely the random matrices and the matrices related to nested hexagons problem. For the later, we considered the particular case $a=4$.

\begin{figure}[!tbp]
  \centering
  \subfloat[Algorithm~\ref{SCCAE-NMF} for \eqref{optiprob2}]{\includegraphics[width=0.49\textwidth]{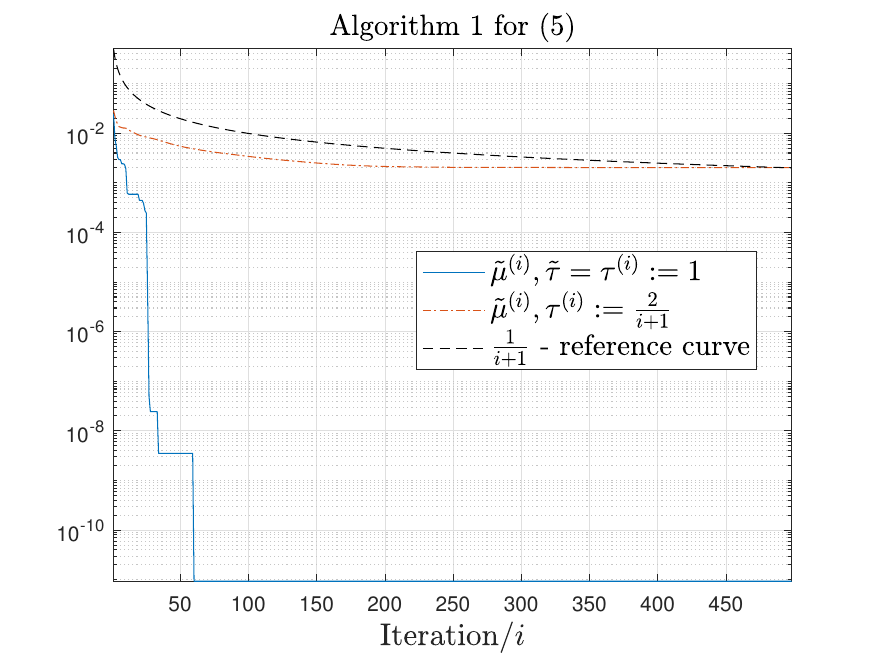}\label{fig:f1append1rand}}
  \hfill
  \subfloat[Algorithm~\ref{SCCAE-NMF} for \eqref{optiprob4}]{\includegraphics[width=0.49\textwidth]{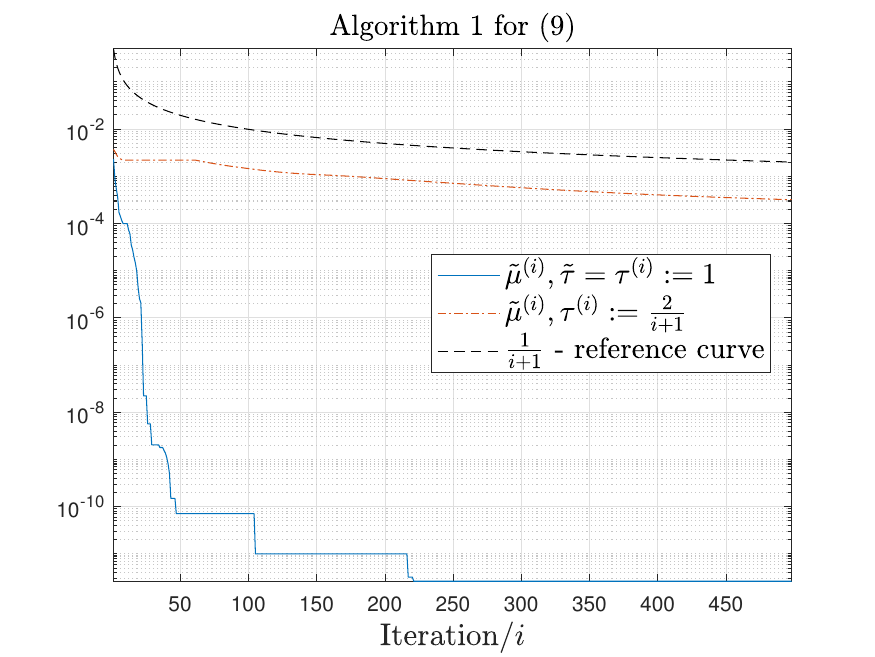}\label{fig:fappend1rand}}
  \caption{Evolution of the minimum FW gap computed along iterations by Algorithm~\ref{SCCAE-NMF} for \eqref{optiprob2} (left) and \eqref{optiprob4} (right) for a randomly geenrated matrix.}\label{fig:emprivalappendix1}
\end{figure}

\begin{figure}[!tbp]
  \centering
  \subfloat[Algorithm~\ref{SCCAE-NMF} for \eqref{optiprob2}]{\includegraphics[width=0.49\textwidth]{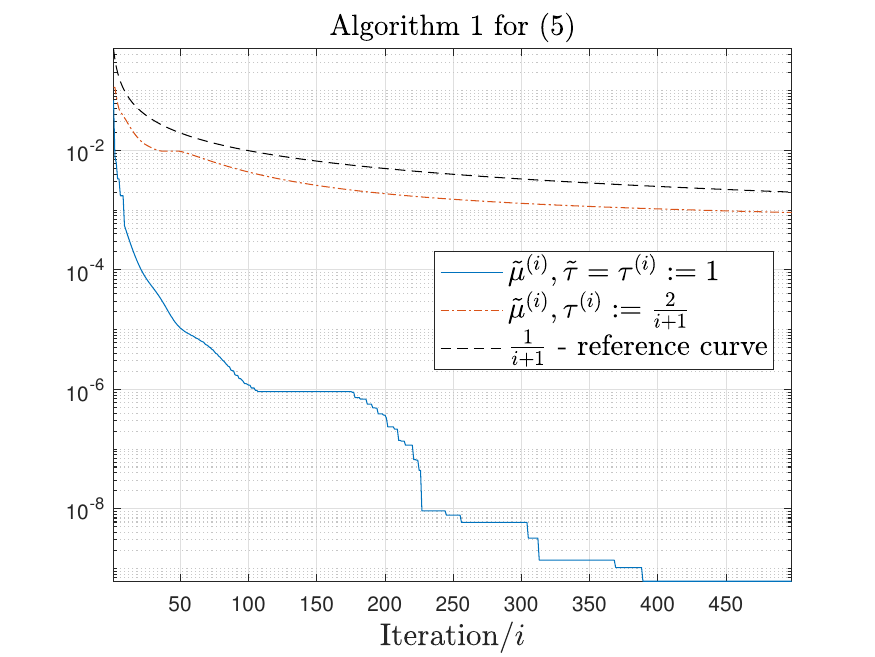}\label{fig:f1append1nest}}
  \hfill
  \subfloat[Algorithm~\ref{SCCAE-NMF} for \eqref{optiprob4}]{\includegraphics[width=0.49\textwidth]{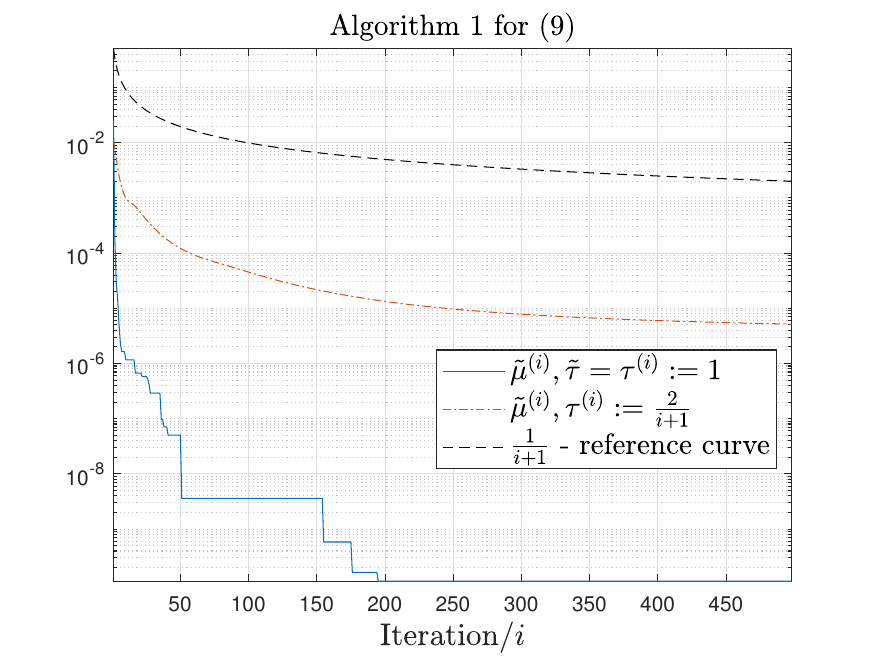}\label{fig:f2append1nest}}
  \caption{Evolution of the minimum FW gap computed along iterations by Algorithm~\ref{SCCAE-NMF} for \eqref{optiprob2} (left) and \eqref{optiprob4} (right) for the matrix ${V_{a=4}}$.}\label{fig:emprivalappendix2}
\end{figure}

\end{document}